\documentclass[12pt]{amsart}
\usepackage{amsmath,amssymb,amscd,amsfonts}
\usepackage{mathtools}
\usepackage{layout}
\usepackage{enumerate}
\usepackage{anysize}
\usepackage{tikz-cd}
\usepackage{scalerel}
\marginsize{25mm}{20mm}{19mm}{30mm}

\usepackage{amsmath,amssymb,amscd,endnotes,graphicx}
\usepackage{mathrsfs}  
\usepackage[mathscr]{euscript}

\usepackage{graphics}
\newtheorem*{theorem*}{Theorem}

\newtheorem{theorem}{Theorem}
\newtheorem{lemma}{Lemma}
\newtheorem{prop}{Proposition}
\theoremstyle{definition}
\newtheorem{definition}{Definition}

\theoremstyle{remark}

\numberwithin{equation}{section}

\newcommand{\Z}{\mathbb{Z}}

\newcommand{\ccirc}{\kern0.5ex\vcenter{\hbox{$\scriptstyle\circ$}}\kern0.5ex}

\def\deg{\operatorname{deg}}
\def\det{\operatorname{det}}

\newcommand{\s}{\sigma}

\def\Z{\mathbb Z}

\def\Q{\mathbb Q}
\def\Qc{\mathcal Q}
\def\Cc{\mathcal C}

\def\Z{\mathbb Z}

\def\Q{\mathbb Q}
\def\Qc{\mathcal Q}

\def\J{\mathcal J}

\def\a{\alpha}
\def\b{\beta}
\def\D{\Delta}
\def\s{\sigma}

\def\g{\gamma}

\newcommand{\defeq}{\mathrel{\mathop:}=}
\newcommand\sbullet[1][.5]{\mathbin{\vcenter{\hbox{\scalebox{#1}{$\bullet$}}}}}

\begin{document}

%\date{5/8/24}
%\title[On two papers of Weil]{On two papers of Weil}
\title[On the arithmetic of polynomials over  a number field ]{On the arithmetic of polynomials over  a number field }
\author{W. Duke}
\address{UCLA Mathematics Department,
Box 951555, Los Angeles, CA 90095-1555} \email{wdduke@g.ucla.edu}
%\thanks{Research supported by NSF grant DMS 1701638 and  Simons Foundation Award Number 554649.}
%
\begin{abstract}Counterparts of several classical results of number theory   are proven for the ring of polynomials with coefficients  in a number field.
A theorem of Milnor that determines the Witt ring of a function field is applied to prove an analogue of Gauss's principal genus theorem for binary quadratic forms with polynomial coefficients. 
 This is used to help understand when and why quadratic reciprocity fails in these
polynomial rings.
Another application  is a count  of  the number of cyclic subgroups whose  order is divisible by four in the primary decomposition of the   torsion subgroup of the Jacobian of certain hyperelliptic curves. 
Invariant theory is applied to prove an  analogue of a classical theorem of Fueter to give criteria for an elliptic curve with a polynomial discriminant and zero $j$-invariant to have no affine  points
over the associated function field.

\end{abstract}
\maketitle

%\begin{center}
%  \small\emph{``The most remarkable ternary form is $x^2-yz$" ... Eisenstein}\footnote{This is an annotation in Eisenstein's copy of the Disquisitiones. See \cite{Weil}. }
%\end{center}

%\begin{center}
%  \small\emph{Procreare jucundum, sed parturire molestum}
%%\footnote{See \cite{Weil}. }
%\end{center}

%\begin{center}
%  \small\emph{``The solution which Gauss has given of it ... depends on principles which are concealed (as is frequently the case in the Disquisitiones Arithmeti\ae) by the synthetical form in which he has expressed it"--Smith }
%\end{center}

\section{Introduction}

Let $k$ be a  field.
As a Euclidean domain with respect to polynomial degree, the ring $k[t]$ shares some number theoretic properties with the integers $\Z$, since it has a  division algorithm and unique factorization into primes (monic irreducibles). 
When $k$ is finite the basic arithmetic theory of $k[t]$ and extensions of its field of fractions $k(t)$ is very well developed.
 %Often the term ``function field'' is reserved this $k(t)$ and its extensions. 
When  $k$ is a number field,  this theory is also an active and attractive area of research (see e.g. \cite{Za0} and its references). 
The aim of this paper is to prove some new counterparts of several old and very well-known results of classical number theory when $\Z$ is replaced by  $k[t]$, where $k$ is a number field.

Given $p,q\in k[t]$ with $p(t)$ prime,  write $\big(\tfrac{q}{p}\big)=1$ if  $p\nmid q$ and $q(t)$ is a square modulo  $p(t)$. 
 If $k$ is finite, after Artin \cite{Ar} quadratic reciprocity holds between any  pair of distinct primes $p$ and $q$, meaning
that 
\begin{equation}\label{art}
\big(\tfrac{q}{p}\big)=1 \;\;\text{if and only if}\;\;\big(\tfrac{p^*}{q}\big)=1,
\end{equation}
where $p^*=(-1)^{|p|}p$ with $|p| \defeq \deg p$  (see also \cite{Ros}).
When $k$ is a number field, it is easy to show that (\ref{art}) does not always  hold. 
 The first problem we want to address  is to characterize those $q$ with odd degree for which
  $\big(\tfrac{q}{p}\big)=1$ implies $\big(\tfrac{p^*}{q}\big)=1$
{\it for all primes} $p\neq q$.
Our approach is guided by Gauss's second proof of classical quadratic reciprocity  in \cite[Art. 262]{Ga0} (or \cite{Ga}), which  is an application of the genus theory of binary quadratic forms.  That proof relies on the finiteness of the class group.
For any square-free $D\in k[t]$ of odd degree, 
let $\J_D$ (or $\J_D(k)$ if $k$ is not understood) be the Mordell-Weil group of $k$-rational divisor classes of degree zero on the hyperelliptic curve over $k$ determined by 
$
s^2=D(t).
$
Naturally, $\J_D$  plays the role of the class group.  After Mordell-Weil,  $\J_D$ is a finitely generated abelian group that  might,  or might not be,  finite.  
   \begin{theorem}\label{t0} Let $k$ be a number field
and $q(t)\in k[t]$ be a fixed prime with odd degree.   
 Then the following holds  if and only if $\J_q$ is finite:
 \begin{equation}\label{eqt1}\big(\tfrac{q}{p}\big)=1 \;\;\text{implies}\;\;\big(\tfrac{p^*}{q}\big)=1\;\;\text{for all primes $p(t)\neq q(t)$}.\end{equation}
\end{theorem}
In words, whether or not $\J_q$  is finite determines whether or not  (one way)  quadratic reciprocity holds for a prime $q\in k[t]$ of odd degree, when  $k$ is a number field.

It can happen that   (\ref{eqt1}) holds for a prime $q(t)\in k[t]$ of odd degree over $k$ and all primes $p\in k[t]$ with $p\neq q$, and where $q$ is still irreducible over a finite extension $k'/k$
but  (1.2) fails for some $p\in k'[t]$. In this case by Theorem \ref{t0}  the rank of $\J_q(k)$  is zero while that of $\J_q(k')$ is at least one.

\subsubsection*{Example}
The polynomial  $q(t)=t^3+16\in \Q[t]$ is a prime of odd degree for which  (\ref{eqt1}) holds for any prime $p\neq q$, since  $s^2=q(t)$ has rank 0 as an elliptic curve over $\Q$. 
Now  $q$ is still irreducible over $k'=\Q(\sqrt{5})$ but  (\ref{eqt1}) fails for $p(t)=t-2(1+\sqrt{5})$. 
This follows since  $q(2+2\sqrt{5})\in k'^2$
so  $(\tfrac{q}{p})=1$, but it can be shown that
 $2+2\sqrt{5}- \a \not\in K^{\sbullet 2}$, where $\a^3+16=0$ and $K=\Q(\sqrt{5},\a)=k'(\a),$ so $(\tfrac{-p}{q})\neq 1$. 
In fact, the curve $s^2=q(t)$ has rank 1 over $k'$ with infinite part generated by $(2+2\sqrt{5},8+4\sqrt{5})$.

\bigskip
%After identifying $\mathcal{J}_q$ with a class group of forms over $k[t]$,
%the proof of the ``if" part of Theorem \ref{t0} follows Gauss's second proof quite closely.  The ``only if'' part   is harder and  uses an analogue of Gauss's principal genus theorem.

For a fundamental discriminant $D\in \Z$, the class  group of binary quadratic forms  over $\Z$ with discriminant $D$ can be identified with the 
(narrow)  divisor class group $H$  of the quadratic field $\Q(\sqrt{D})$.  
 For any finitely generated  abelian group $G$, let  $e_n=e_n(G)$ be the number of subgroups of order $2^m$, where $m\geq n$,  in a representation of the torsion part of $G$ as a direct product of cyclic subgroups of prime power order.  After Gauss,  $2^{e_1(H)}$ is the number of (unordered) factorizations $D=D_1D_2$ into fundamental discriminants $D_1,D_2$.  An elegant supplement was published  in \cite{RR}.  It may be viewed as a consequence of the principal genus theorem of Gauss. 
\begin{theorem*}[Redei-Reichardt]
 There are  $2^{e_2(H)}$  factorizations $D=D_1D_2$ into fundamental discriminants 
 that satisfy $(\tfrac{D_1}{p})=1$ for each $p|D_2$ and $(\tfrac{D_2}{p})=1$ for each $p|D_1$. 
\end{theorem*}

I will derive a parallel result from a $k[t]$-version of the principal genus theorem.

\smallskip
Again, let  $\J=\J_D$ be the Mordell-Weil group of the hyperelliptic curve over a number field $k$ given by 
 $
 s^2=D(t).
 $
As is well-known (and also follows from Propositions \ref{amb} and \ref{the1} below), for  $D\in k[t]$  monic and  square-free  of odd degree,   there are $2^{e_1(\J)}$  factorizations $D=D_1D_2$ into co-prime monic polynomials, where $D_2$ has odd degree.  Our second theorem determines $e_2(\J)$.
 \begin{theorem}\label{t1} Let $D\in k[t]$ be  square-free and monic of odd degree. Then there are $2^{e_2(\J)}$ factorizations $D=D_1D_2$ into monic polynomials each  in $k[t]$ with $D_2$ of odd degree,  where  $(\tfrac{D_1}{p})=1$ for each prime $p|D_2$ and  $(\tfrac{-D_2}{p})=1$ for each prime $p|D_1.$
\end{theorem}

\subsubsection*{Example}
For any number field $k$ 
suppose that $f\in k[t]$ is monic of degree $\geq 1$ and  $f(0)=\a\neq 0$ and that $D_1(t)\defeq f^2(t)+t$ is square-free. Set $D_2(t)=t$ and
let \[D(t)=D_1D_2=(f^2(t)+t)t.\] Then  $D$ is monic, square-free and has odd degree.
Clearly $(\tfrac{D_1}{t})=1$ since $D_1(0)=a^2\in k^{\sbullet 2}.$  Also, $(\tfrac{-t}{p})=1$ for any prime $p|D_1$ since
\[
-t\equiv f^2(t) \pmod{p(t)}.
\]
Hence by Theorem \ref{t1} the Mordell-Weil group of the genus $|f|$ curve $s^2=t(f^2(t)+t)$ over $k$ has points of order 4.

\bigskip
The non-existence of classes of order 3 in the class group $H$ of $\Q(\sqrt{D})$ has a consequence for the rational points on the elliptic curve 
$
y^2=x^3+D,
$
at least for some $D$. The next classical result follows from a  more general statement proven in \cite{Fue}.
\begin{theorem*}[Fueter]\label{f}
 Suppose that $D\neq 1$ is  square-free. Then the equation $y^2=x^3+D$ either has  an infinity of  solutions $x,y\in \Q$  or it has none. It has none if
 $D<0$ satisfies $D\equiv 2\pmod{9}$,   $D\not \equiv 1\pmod{4}$ and $H$ has no elements of order 3.
\end{theorem*}

I will give an analogue of this result as well for the rational points on the elliptic curve over $k(t)$ determined by  $y^2=x^3+D(t)$, where $D\in k[t]$ is square-free and non-constant.
This equation can also be considered to give an elliptic surface over $k$ where, in general, it has many more points.
For $D\in k[t]$ let 
\begin{equation}\label{wide}
D'(t)=-27D(-3t).
\end{equation}

\begin{theorem}\label{t2} Suppose that $D\in k[t]$ is square-free and non-constant, where $k$ is a number field.  Then either 
the equation \[y^2=x^3+D(t)\]  has an infinity of solutions $x,y\in k(t)$ or it has none. It has none 
if $D$ is monic of  odd degree and neither  $\mathcal{J}_D$ nor $\mathcal{J}_{D'}$  has points of order 3.
\end{theorem}

\subsubsection*{Examples}
i) Suppose that $a,b\in k$ satisfy  $4a^3+27b^2\neq 0$ and that 
\begin{equation}\label{is}
3t^4+6at^2+12bt-a^2=0
\end{equation}
has no roots  $t\in k$. 
Then the equation \begin{equation}\label{ex21}  y^2=x^3+t^3 +at+b\end{equation}has no solutions $x,y \in k(t)$.   To see this, note that (\ref{is}) determines the $t$-coordinates of points of order 3 on the elliptic curve $\mathcal{J}_D: s^2=t^3 +at+b$
and the change of variables $t\mapsto -\frac{t}{3}$  gives the corresponding equation for  $s^2=t^3 +9at-27b$,  which is isomorphic to  $\mathcal{J}_{D'}$.  Thus,
 if (\ref{is}) has no rational roots then neither $\mathcal{J}_D$ nor $\mathcal{J}_{D'}$ 
 has rational points of order 3 and we may apply Theorem \ref{t2}.
Suppose that $k=\Q$. Using (\ref{is}), an easy counting  argument  shows that the number $N(X)$ of admissible $(a,b)\in \Z$ with \[\max (4|a|^3,27|b|^2)\leq X\] and such that (\ref{ex21}) has solutions  $x,y\in \Q(t)$,
 satisfies   $N(X) \ll X^{\frac{1}{2}}$. Therefore equations (\ref{ex21}) with $a,b\in \Z$ and  solutions $(x,y)\in \Q(t)^2$ are rare.
%On the other hand, it is also easy to show by construction that there exist $\gg X^{\frac{1}{2}}$ such equations with rational solutions. For instance, 

\smallskip
ii) The curve $y^2=x^3+t^3-3$ with $D(t)=t^3-3$ over $\Q(t)$ contains the point $(\tfrac{4-t^3}{t^2},\tfrac{3t^3-8}{t^3})$.  Now $\mathcal{J}_{D}$ has trivial torsion, while 
$\mathcal{J}_{D'}$  has points of order 3.  Thus it is not enough in the second statement of Theorem \ref{t2} to assume only that  $\mathcal{J}_D(k)$  has no points of order 3
 to conclude that $y^2=x^3+D(t)$ has no solutions. Similarly, it is not enough to assume only that $\mathcal{J}_{D'}(k)$ has no points of order 3. 
 
 \smallskip
  \subsubsection*{Acknowledgements}
  I am grateful to David Rohrlich and Umberto Zannier for their valuable comments. 

\section{Overview and discussion of proofs }
  %  \subsubsection*{Theorems \ref{t0} and \ref{t1}}
To prove the theorems of this paper, binary forms with polynomial coefficients provide an efficient and convenient tool. 
The next section \S \ref{s3} ends with the proofs of Theorems \ref{t0} and \ref{t1}. 
We use a set of binary quadratic forms with coefficients in $k[t]$ that are analogous to positive definite binary quadratic forms with integral coefficients. 
They split into a set of classes $\mathcal{C}_D$ of forms with discriminant $D$. 
 An adaptation of Dirichlet-Dedekind's version of Gaussian composition  makes $\Cc_D$    into an abelian group.  
 Following  Gauss, we  count  classes of order two in terms of ambiguous forms.   
 We define genus characters using  certain well-known Steinberg symbols studied in \cite{Mil}.   In the classical case, to prove quadratic reciprocity it is enough to know,  in addition to the number of  classes of order two,  that the class number is finite and that a genus character is trivial when  evaluated at a square.  
 For us, the corresponding argument is enough to prove the ``if" part of Theorem \ref{t0}.  We actually prove the theorems with $\J_D(k)$ replaced by $\Cc_D$ in the statements, and later show that $\J_D(k)$ is isomorphic to $\Cc_D.$
 
For the proof of the converse   we need  an analogue of Gauss's principal genus theorem.  This result identifies $\Cc_D^2$ with the classes killed by all genus characters.
 Here a new difficulty arises since  there are infinitely many polynomials of bounded degree. The fact that there are only finitely many integers of bounded norm is  behind most proofs of the classical principal genus theorem, such as    Gauss's,  that uses reduction of  indefinite ternary quadratic forms, and proofs  using $L$-functions or  Legendre's theorem.
To overcome this difficulty, we  apply  the determination of the Witt ring of  $k(t)$ given by Milnor \cite{Mil2} and formulated in Proposition \ref{mh}. This result  implies a kind of Hasse principle (see Proposition \ref{has})
 that allows us to identify genera with rational classes (i.e. $k(t)$-classes) of forms and prove the genus theorem.
   It is to be observed  that the theorem of Milnor is not formulated in terms of local fields but in terms of their residue fields, which for us are number fields. 
 
  The proof of the ``only if" part of Theorem \ref{t0} also requires  that if $\Cc_D$ is infinite, it must contain a class that is not a  square. For this it is enough to know that $\Cc_D$ is finitely generated,  which is  a consequence of the Mordell-Weil theorem.  Toward the end of \S \ref{s3} we show that  $\Cc_D$ is isomorphic to the Mordell-Weil group $\J_D$.
This result  is based in part on work of Jacobi \cite{Jac},  Mumford \cite{Mum} and David Cantor  \cite{Can1}. Cantor  translated divisor arithmetic into computations with polynomials
 that also occur in the  composition formulas.   His algorithms have  been extensively applied when $k$ is a finite field. 
For my purpose, I need to make the  relation between divisor classes and the classes of binary quadratic forms that I have defined, completely explicit. 
In terms of quadratic forms, the weak Mordell-Weil theorem 
is equivalent to the finiteness of the number of genera. In view of this, the result we give in Proposition \ref{mt0}  that forms are in the same genus if and only if they are rationally equivalent,  is of independent interest.

A third ingredient  in the proof of the  ``only if" part of Theorem \ref{t0} is an application of Hilbert's irreducibility theorem over $k.$
Thanks to  this result,  it  is not hard to make a  binary quadratic form in $\Qc_D$  represent primes. 

 Concerning Theorem \ref{t1},  in view of  \cite[Art. 305-307]{Ga0} it seems likely that  Gauss was aware that the principal genus theorem gives information about  elements of order 4 in the class group of binary forms over $\Z$,  along the lines of the theorem of Redei and Reichardt that we stated. Such a proof is carried out  in  \cite[p.163]{Ven}.
 Once we have the principal genus theorem over $k[t]$ and have characterized the classes of order two,   Theorem \ref{t1} follows similarly.

 The proof of the second statement in Theorem \ref{t2} is based on the theory of binary cubic forms with coefficients in $k[t]$;  the aspects of this theory that we need are given in \S \ref{s4}.   That binary cubic forms over $\Z$ give information about classes of binary quadratic forms of order 3 was discovered by Eisenstein \cite{Eis}. 
%It is perhaps surprising that the second part of Fueter's theorem was not found  earlier than 1930, since the syzygy giving the equation $E_D: y^2=x^3+D$ was already well known and being applied arithmetically by Eisenstein \cite{Eis}, Arndt \cite{Arn}, Cayley \cite{Cay1} and Hermite \cite{Her}. 
We make  essential use of a $k(t)$-version of a familiar pair of dual 3-isogenies connecting the elliptic curves $E_D$ and $E_{-27D}$,  one  that  was employed by Fueter (see \cite[p.71.]{Fue}).
The non-existence part of Fueter's theorem was extended and parts of his proof simplified by Mordell (see \cite{Mor1} and \cite[Chap. 26]{Mor}).   These proofs of Fueter and Mordell use algebraic number theory.  As will become clear, it is also possible to prove such results using classical binary cubic forms.\footnote {On the other hand, it would be natural give a  proof of  the second statement of Theorem \ref{t2} along the lines of those of Fueter and Mordell using quadratic extensions of $k(t).$ In fact, Zannier has kindly shown me such a proof of the first statement.}  
% Once the requisite properties of binary cubics are known, in some ways the proof over $k(t)$ is simpler  than that over $\Q$.  One explanation for this is  the fact that 2 and 3 are not primes in $k[t].$
 Well-known results from algebraic number theory allow one to  assume in the statement of  such theorems over $\Q$  the class number condition for  only one of $\Q(\sqrt{D})$ and $\Q(\sqrt{-3D})$ (c.f. \cite{Sch}). 
The corresponding simplification can not be made in general in  Theorem \ref{t2}, as Example ii) below its statement shows.

 A consequence of some parts of the  proof of Theorem \ref{t2}  is  the finiteness of the number of classes of binary cubic  forms with coefficients in $k[t]$ and with a fixed discriminant $4D$, where $D$  is  monic, square-free with odd degree (see Proposition \ref{ii}).
  
The changes needed to include discriminants  $D\in k[t]$ with even degree in our results lead to  interesting problems that are not addressed in this paper.

\section{Binary quadratic forms over $k[t]$}\label{s3}

%If $|a|$ is even, then  $a$ is positive if and only if it has an approximate square-root   $s\in k[t]$  in that   \begin{equation}\label{sqrt}|s^2-a|<\tfrac{1}{2}\deg a.\end{equation} If such an $s$ exists it is unique up to sign.  

The theory of  binary quadratic forms over $k[t]$ can be developed in a way that is  quite similar to that over $\Z$, especially when the discriminant has odd degree, where the forms are similar to positive definite ones. 
%An excellent overview of the classical theory of positive definite binary quadratic forms is given in  \cite{Cox}.

I assume throughout that $k$ is a fixed number field. However, 
except for the application of the Mordell-Weil theorem to $\J_D$,  most of the results presented here remain valid when $k$ is any perfect field of characteristic not 2.

   \subsection*{Preliminaries} 
The gcd of a finite set  of polynomials in $k[t]$ is the monic polynomial of largest degree that divides each. Two polynomials are coprime if their gcd is 1.  If $c=\gcd (a,b)$ then there are $x,y\in k[t]$ with $c=ax+by.$
For $a\in k[t]$ the inequality  $|a|<0$ is equivalent to  $a=0.$

Suppose that $D\in k[t]$ is square-free. 
 Associated to  $Q=\left( \begin{smallmatrix}a& b\\b&c \end{smallmatrix}\right)$  with $a,b,c\in k[t]$  is the binary quadratic form
\[
Q(x,y)=(a,b,c)=a(t)x^2+2b(t)xy+c(t)y^2
\]
with  discriminant ($=$ Gaussian determinant) $D=b^2-ac.$ 
  Note that any form $Q=(a,b,c)$ with $b^2-ac=D$ is primitive in that $\gcd(a,b,c)=1.$
  Let $M\in \mathrm{SL}_2(k[t])$ act on $Q$ by
  \begin{equation}\label{e2}Q|M\defeq MQM^t. \end{equation}
For $M=\left( \begin{smallmatrix}m& r\\n&s \end{smallmatrix}\right)$  we have   $Q|M(x,y) =Q(mx+ny,rx+sy)=Q'(x,y)$ and  the coefficients of  $Q'=(a',b',c')$  are given by 
\begin{align}\label{eq00}
a'=&am^2+2bmn+cn^2\\\nonumber
b'=&amr+b(ms+rn)+cns\\\nonumber
c'=&ar^2+2brs+cs^2.
\end{align}
Clearly $b'^2-a'c'=D.$ The two forms $Q$ and $Q'$ are said to be equivalent, written $Q\sim Q'.$

Say that $n\in k[t]$ is  {\it represented} by $Q$ if $Q(x,y)=n$ for  $x,y\in k[t]$
and that $n$ is {\it properly represented} by $Q$ if $Q(x,y)=n$ for coprime $x,y\in k[t].$ Equivalent forms represent and properly represent the same polynomials.

The integral version of the following useful  result was given by Dirichlet \cite[\S 60]{Dir2}.   The proof for $k[t]$ is similar.
\begin{lemma}\label{gl}
Let $Q=(a,b,c)$ and $Q'=(a',b',c')$,  both with discriminant $D$, be given. They are equivalent if 
there exist $m,n,r,s\in k[t]$ with 
\begin{align}\label{eq1}
&a'=am^2+2bmn+cn^2,\\\nonumber
&(b+b')n+am=a's\;\;\text{and}\\\nonumber
&(b'-b)m-cn=a'r,
\end{align}
in which case  $Q|M=Q'$  with $M=\left( \begin{smallmatrix}m& r\\n&s \end{smallmatrix}\right)\in \mathrm{SL}_2(k[t])$.
Conversely, if $Q,Q'$  satisfy $Q'=Q|M$ with this $M$,  then (\ref{eq1}) holds. 

\end{lemma}

A proof of the next standard result may be found in  \cite[Thm VII.3]{New}).
\begin{lemma}\label{sl}
        $\mathrm{SL}_2(k[t])$ is generated by
\[
C_\a=\left(\begin{smallmatrix}
   \a& 0\\
    0& \a^{-1}
    \end{smallmatrix}\right),
T_n=\left(\begin{smallmatrix}
   1& n(t)\\
    0& 1
    \end{smallmatrix}\right)
\;\;\;\;\mathrm{and}\;\;\;S=\left(\begin{smallmatrix}
   0& -1\\
    1& \;0
\end{smallmatrix}\right)
\]
for $n\in k[t]$ and $\a\in k^{\sbullet}.$
\end{lemma}
We have the formulas
\begin{align}\label{eqf1}
&(a,b,c)|C_\a=(\a^2a,b,\a^{-2}c)\\\nonumber
&(a,b,c)|T_n=(a,b+na,c+2nb+n^2a)\\\nonumber
&(a,b,c)|S=(c,-b,a).
\end{align}

The following Lemma is easily verified.

 \begin{lemma}\label{ls1}
 Suppose that $D\in k[t]$ is square-free and that $Q=(a,b,c)$ has $D=b^2-ac.$
The group of  automorphs  of $Q$, that is those $M \in \mathrm{SL}_2(k[t])$ with $Q|M=Q$, are all of the form 
\begin{equation}\label{M}
M=\pm \left(\begin{smallmatrix}
 x-by& ay\\
   - cy& x+by
    \end{smallmatrix}\right),\end{equation}
for $x,y \in k[t]$ solutions of the Pell equation 
 \begin{equation}\label{pell}
   x^2-D(t)y^2=1.
    \end{equation}
  Transformations $M_1$ with $Q|M_1=(-a,b,-c)$ exist if and only if the negative Pell equation 
    \begin{equation}\label{npell}
   x^2-D(t)y^2=-1
    \end{equation}
has solutions, in which case each such $M_1$ is given by
\begin{equation*}\label{M2}
M_1=\pm \left(\begin{smallmatrix}
   x-by& ay\\
   cy& -x-by
    \end{smallmatrix}\right).\end{equation*}
    For a solution $x,y$ of (\ref{npell}), the  transformation $M$ from (\ref{M}) has $\det M=-1$ and satisfies $Q|M=-Q.$
    
    \end{lemma}
 Solutions with $y\neq 0$ to (\ref{pell}) can exist only for even degree $D$  and then only in special cases.   The problem of solving polynomial Pell equations has a long history 
 going back to Abel and Chebyshev. See \cite{Pla},  \cite{Za} for more information.

 \begin{lemma}\label{repp}
  Let $Q$ have  discriminant $D$.
  
  i)  If $a'\in k[t]$ is properly represented by $Q$ then there exists  $Q'=(a',b',c')$ that is equivalent to $Q$, where $|b'|<|a'|.$
  
  ii) $Q$ can properly represent a polynomial $n\in k[t]$ with $|n|\leq \tfrac{1}{2}|D|.$
  
  iii) $Q$ can properly represent some $n\in k[t]$ that is prime to any fixed $f\in k[t]$.
  
  \end{lemma}
  \begin{proof}
  
i)  If $Q(m,n)=a'$ where $m,n$ are coprime then there exist $r,s\in k[t]$ with \[M=\left( \begin{smallmatrix}m& r\\n&s \end{smallmatrix}\right)\in \mathrm{SL}_2(k[t]).\]
Thus $Q|M=(a',*,*)$. Using $T_f$ from (\ref{eqf1}) with an appropriate $f$ found by the division algorithm, we can force  $|b'|<|a'|.$

ii) Let $n\in k[t]$ be a polynomial of minimal degree represented by $Q$. The representation is proper. By i)  we may assume that  $Q=(n,b,c)$, where $|b|<|n|.$ 
Since $c$ is properly represented by $Q$ we must have $|b|<|n|\leq |c|$, which implies that $2|n|\leq |nc|=|D|.$

iii) Let $p_1\in k[t]$ be the product of all primes dividing $a,c$ and $f$. Let $p_2$ be the product of all primes dividing $a$ and $f$ but not $c$. Let $p_3$ be the product of all primes dividing $c$ and $f$ but not $a$. 
Let $p_4$ be the product of all remaining primes dividing  $f$. Then  
\[
n=Q(p_2,p_3p_4)=a(p_2)^2+2bp_2(p_3p_4)+c(p_3p_4)^2
\]
is prime to $f$ and $\gcd(p_2,p_3p_4)=1.$
 \end{proof}

\subsection*{Forms with a negative discriminant}

The following definition streamlines the  formulation of theorems and is suggestive, if unconventional.
\begin{definition}\label{de1}
Say that $n \in k[t]$ is {\it positive} (and write $n\succ0$) if the leading coefficient of $n^*=(-1)^{|n|}n$ is in $k^{\sbullet 2}$ and {\it negative} ($n\prec0$) if $-n$ is positive.
 Here $|n|=\deg n.$
 \end{definition}
 The product of two positive or two negative polynomials is positive and the product of one positive and one negative polynomial is negative.

\begin{definition}
When $D\in k[t]$ is negative and  square-free with $|D|$   odd,
 let  $\Qc_D$  denote the set of those  $Q=(a,b,c)$ with  discriminant $D$, for which $a\succ0$.
\end{definition}
\smallskip
 Observe that $\Qc_D$  is non-empty, since it contains the principal form $(1,0,-D).$   The following lemma shows that the action $Q\mapsto Q|M$ splits $\Qc_D$ into a set of equivalence classes and that 
the forms in $\Qc_D$  are similar to  positive definite integral forms. 

\begin{lemma}\label{cl}

Suppose that  $D\in k[t]$ is negative and square-free with $|D|$ odd.    If  $Q=(a,b,c)\in \Qc_D$,  then $Q|M\in \Qc_D$ for  $M\in \mathrm{SL}_2(k[t])$. In particular, any $n$ properly represented by $Q$ is positive.
\end{lemma}
\begin{proof}
From $D=b^2-ac$ and $a\succ0,$  we will show  that $c\succ0.$ Clearly $c\neq 0$. Next use that the leading coefficient of $D$ is in $k^{\sbullet 2}$ with $|D|$  odd to check the different combinations of parities of the degrees of $a$ and $c$. If they are both even or both odd the leading terms of $b^2$ and $ac$ must cancel. Otherwise the leading term of $-ac$ must be in $k^{\sbullet 2}.$

The first statement for any $M$  now follows from  Lemma \ref{sl} and the formulas (\ref{eqf1}).
The second statement follows from the first and i) of Lemma \ref{repp}.

\end{proof}
\begin{definition} For $D\in k[t]$ negative and  square-free with $|D|$  odd, let
$\Cc_D$ be the set of all $\mathrm{SL}_2(k[t])$-classes of forms $Q\in \Qc_D.$
\end{definition}

We have the following analogue of Lagrangian reduction.  
\begin{prop}\label{min}
   Each class $C\in \Cc_D$  contains a unique form $Q=(a,b,c)$ with  \[|b|<|a|<\tfrac{1}{2}|D|,\] where  $a^*$ is monic.
   \end{prop}
\begin{proof}
The existence follows from Lemmas  \ref{repp} and \ref{cl},  after applying the appropriate $C_\a$ from (\ref{eqf1}) to achieve the monic condition. Here we use that $|D|$ is odd. 

For uniqueness, suppose that  $Q'=(a',b',c')\in C$ has $|b'|<|a'|<\tfrac{1}{2}|D|$ with $\pm a'$ monic.  
By completing the square, \[
aa'=(am+bn)^2-Dn^2,
\]
where $Q'=Q|M$ with
$M=\left( \begin{smallmatrix}m& r\\n&s \end{smallmatrix}\right)\in \mathrm{SL}_2(k[t])$. 
Therefore \[|D|>|a|+|a'|=|(am+bn)^2-Dn^2|.\]
This implies that $n=0$ and so from the first equation of (\ref{eq00}) we have that $am^2=a'.$ By symmetry we  deduce that $m=s=1$ and $a=a'$.   The second equation of (\ref{eq00}) yields
$b'=b+ra$, which implies that $r=0$ since \[|b'-b|<|a'|=|a|.\] Thus $Q=Q'.$
\end{proof}

\subsection*{Class group and ambiguous classes }

A straightforward modification  of Dirichlet and Dedekind's  \cite{Dir2} version of Gaussian composition can be defined for classes  in $\mathcal{C}_D$, making it into an abelian group.
We do, however, have to account for our condition on $(a,b,c)\in \Qc_D$ that $a\succ0$, but this leads to no problem since if $a\succ0$ and $a'\succ0$ then $aa'\succ0$.
Say  two forms $Q=(a,b,c),Q'=(a',b',c')\in \Qc_D $ are {\it concordant} if
\[
\gcd(a,a', b+b')=1.
\]
By Lemma \ref{repp}, any pair of classes will contain a pair of concordant forms. 
The following result can be proven by adapting \S 145-\S148 of  \cite{Dir2} (see also  \cite[pp.129--132]{Ven}). Use is made of  Lemma \ref{gl} to show that the product  of classes is well-defined.

\begin{prop}\label{clg}
Let $D(t)\in k[t]$ be negative,  square-free with $|D|$ odd and   $C,C'\in \Cc_D$.  

\bigskip
i) If $(a,b,c)\in C$ and $(a',b',c')\in C'$ are concordant,  there exist  $e,f\in k[t]$ such that  $(a,e,a'f)\in C$ and $(a',e,af)\in C'$.

\bigskip
ii) The class $C''\in \Cc_D$ of  $(aa',e,f)$ is well-defined and makes $\Cc_D$ into an abelian group through $C\otimes C'=C'',$
where the  identity is the class containing $(1,0,-D)$ and the inverse $C^{-1}$ of the class $C$ is the class of $(a,-b,c)$.

\end{prop}
Letting $Q=(a,e,a'f), Q'=(a',e,af)$ and $Q''=(aa',e,f)$ we have that $Q''$ is directly composed of $Q$ and $Q'$ in the sense of Gauss  \cite[Art. 235]{Ga}:
\[
Q(x,y)Q'(x,y)=Q''(X,Y),\;\;\;\text{where}\;\;(X,Y)=(xx'-fyy',axy'+a'x'y+2eyy').
\]
Together with this and Proposition \ref{clg}, 
arguments  of Gauss \cite[Art. 237--239]{Ga} showing that  composition of classes is well-defined  adapt  to give the following.
\begin{prop}\label{clg2}
\bigskip
Suppose that $Q\in C$ and $Q'\in C'$ where $C,C'\in \Cc_D$. If
\begin{align}\label{gac}
&Q''(X,Y)=Q(x,y)Q'(x',y')\;\;\;\text{with}\\\nonumber(X,Y)=&(a_1xx'+b_1xy'+c_1yx'+d_1yy',a_2xx'+b_2xy'+c_2yx'+d_2yy'),
\end{align}
where  $a_j,b_j,c_j,d_j\in k[t]$ for $j=1,2$ are such that  \[a_1b_2-a_2b_1=Q(1,0)\;\;\text{ and}\;\; a_1c_2-a_2c_1=Q'(1,0),\]
then the class $C''$ of $Q''$ equals $C\otimes C'.$ 
\end{prop}

An important consequence of Proposition \ref{clg} is the following result. 
\begin{prop}\label{pro}Suppose that $Q=(a,b,c)\in C$ and $Q'=(a',b',c')\in C'$, where $C,C'\in \Cc_D$ are not necessarily distinct. 

%i) If $Q$ represents $n$ and $Q'$ represents $n'$ then any form in $C\otimes C'$ will represent $nn'$. 

i) If $a$ and $a'$ are coprime, then any form in  $C\otimes C'$ will properly represent $aa'.$

ii) If  $a$ is prime to $D$, then any form in $C\otimes C$ will properly represent $a^2.$
\end{prop}
\begin{proof}
%Part i) follows from  Proposition \ref{clg2}.  
For i), observe that $Q$ and $Q'$ are then concordant. For ii), if $a$ is prime to $D$ then it is prime to $b$ and so $Q$ and $Q$ are concordant.
\end{proof}

The following easily proven result shows that for our purposes we may assume if convenient that a  negative $D$ is monic.
\begin{lemma}\label{isom}
For any fixed  $\a \in k^{\sbullet},$ the map  $(a,b,c)\mapsto (a, \a b, \a^2c) $ defines a group isomorphism from $\Cc_D$ to $\Cc_{\a^2 D}$.
\end{lemma}

By an {ambiguous form} is  meant  a form of type $(a,0,c).$ Note: over $k[t]$ we need not  consider non-diagonal ambiguous forms, which are necessary over $\Z$.
A class $C\in \Cc_D$ is called {\it ambiguous} if it has order 1 or 2.  

\begin{lemma}\label{cant}
A class of $\Cc_D$ is ambiguous if and only if it contains an ambiguous form.
\end{lemma}
\begin{proof}
Clearly the class of $(a,0,c)$ has order 1 or 2.  For the converse, we adapt an argument of Georg Cantor \cite{Cant} from the classical case.
Suppose that the square of the class of $Q=(a,b,c)$ is the identity.  What is the same,  $(a,b,c)$ is equivalent to $(a,-b,c).$
By Lemma \ref{gl} there are $m,n,r,s\in k[t]$ such that $ms-nr=1$ and where 
\begin{equation}\label{need}
a=am^2+2bmn+cn^2,\;\;\;am=as,\;\;\;\;\;-ar=2bm+cn.
\end{equation}
It follows that $m=s$ so from  $ms-nr=1$ we see that $Q'=(r,m,n)$ has discriminant 1.  By Lemma \ref{repp}, $Q'$ is equivalent to $(\a,0,-\a^{-1})$ for some $\a\in k^{\sbullet}$, which is equivalent to \[(0,-1,0)=(\a,0,-\a^{-1})|\left( \begin{smallmatrix}1& -1/(2\a)\\\a&1/2 \end{smallmatrix}\right).\]
Hence there  will exist $N=\left( \begin{smallmatrix}a_1& a_2\\a_3&a_4 \end{smallmatrix}\right)\in \mathrm{SL}_2(k[t])$ with $(0,-1,0)|N=(r,m,n)$. Thus 
\[
r=-2a_1a_3\;\;\;m=-a_1a_4-a_2a_3\;\;\;n=-2a_2a_4.
\]
Finally,  using these identities with the middle equation of (\ref{eq00}) we get \[(a,b,c)|N^t=(*,-2bm-cn-ar,*),\]
so the last equation of (\ref{need})  shows that $(a,b,c)|N^t=(*,0,*)$  is ambiguous. 
\end{proof}
For $n\in k[t]$, let $\omega(n)$ be the number of distinct primes dividing $n$.

\begin{prop}\label{amb}
The subgroup of $\Cc_D$ of elements of order at most two has order $2^{\omega(D)-1}$.
\end{prop} 

\begin{proof}
 By Proposition \ref{min} and Lemma \ref{cant}, an ambiguous class of $\Cc_D$ contains a unique ambiguous form $(a,0,c)$ where $|a|<|c|$ and $ a^*$ is monic. The number of such forms is  $2^{\omega(D)-1}$.
\end{proof}

%\subsubsection*{Remark} The analogous result of Gauss is more difficult due to the prime 2.
 \subsection*{Principal genus theorem}

 In this subsection I will prove the $k[t]$ version of the principal genus theorem.  
We will use a  Hilbert-type Steinberg symbol to define ``genus characters.''
 A reference with proofs of those properties we give is \cite[p.98]{Mil2}.
 The field extension of $k$ determined by  $F_p= k[t]/(p)$  for a prime $p\in k[t]$, has degree $|p|$ and can be identified with $k(\a)$, where $\a\in\bar{k}$ has $p(\a)=0$,
 where $\bar{k}$ is the algebraic closure of $k$.
Define for $a,b \in k(t)^{\sbullet}$
\begin{equation}\label{sy1}
(a,b)_p = (-1)^{v_p(a)v_p(b)}a^{v_p(b)}b^{v_p(a)}\in F_p^{\sbullet}/F_p^{\sbullet 2},
\end{equation}
where $v_p(a)$ is determined by $a=p^{v_p(a)}u$ with $u$ prime to $p$.
Also, set
\begin{equation}\label{sy2}
(a,b)_\infty = (-1)^{|a||b|}a_0^{|b|}b_0^{|a|}\in k^{\sbullet}/k^{\sbullet 2}.
\end{equation}
For any $p$ (inc. $p=\infty$) and  $a,b \in k(t)^{\sbullet}$ the following hold 
\begin{align}
&(a,b)_p=(b,a)_p\\\label{mult}
&(ab,c)_p=(a,c)_p(b,c)_p\\\label{mul}&
(a,1-a)_p=1, \; a\neq1\\\label{sq}&(a,b^2)_p=1.
\end{align}
We also  have the product formula
\begin{equation}\label{pf}
(a,b)_\infty \prod_{p}N_{F_p/k}(a,b)_p=1,
\end{equation}
where  $N_{F_p/k}:F_p\rightarrow k$ is the norm.

For any $p$ including $p=\infty$ define the genus ``character"  $\chi_p: \mathcal{Q}_D\rightarrow F_p^{\sbullet}/F_p^{\sbullet 2}$ by
\begin{equation}\label{gc}\chi_p(Q)=(n,D)_p,\end{equation}
where  $n$ is any polynomial prime to $p$ that is properly represented by $Q$. 
 \begin{lemma}\label{lwd}
The  value of $\chi_p(Q)$ is well-defined; it is independent of the $n$ chosen. In addition,  $\chi_p(Q)=1$ for any $p\nmid d$, including $p=\infty$.
\end{lemma}
\begin{proof}
We have the identity\footnote{This is not a direct composition; it fails to satisfy the last conditions given in Proposition \ref{clg2}.}
\begin{align}\label{ide1}
Q(x,y)Q(x',y')&=(axx'+b(xy'+x'y)+cyy')^2-D(t)(xy'-x'y)^2.\end{align}
Suppose that  $n$ and $n'$ are prime to $D$ and properly represented by $Q$. 
 For $p|D$ it follows from  (\ref{ide1}), (\ref{sq})  and (\ref{mult}) that $1=(nn',D)_p=(n,D)_p(n',D)_p$.
Suppose that $p|n$.  By  Lemmas \ref{repp} i) and \ref{cl}, we have that $Q$ is properly equivalent to some $(n,b,c)\in \Qc_D$, so $D=b^2-nc$ and $D$ is a square modulo $p$. Thus $(n,D)_p=1$.   For $p=\infty$, since $n\succ0$,   we  have that  $(n,D)_\infty=1$ from (\ref{sy2}).
    For any other $p$, that $(n,D)_p=1$ is immediate from the definition. \end{proof}
Each $\chi_p$  is well defined on the class of $Q$.
It induces a homomorphism from $\Cc_D$ to $F_p^{\sbullet}/F_p^{\sbullet 2}$
 by (\ref{mult}) and i) of Proposition \ref{pro} together with Lemma \ref{repp}. 
The intersection of the 
kernels  of $\chi_p$ for $p|D$ will be called the {\it principal genus} and the cosets of the principal genus the {\it genera}.
Let $\mathcal{G}_D=\mathcal{G}_D(k)$ denote the group of  genera.
 The principal genus  is that genus containing  $(1,0,-D)$. Thus we have
 
 \begin{prop}\label{mt3}
The genera $\mathcal{G}_D$ comprise an abelian group in which every non-identity element is of order  two. Each genus $G$ is characterized by the values $\chi_p(Q)$ for $p|d$, where $Q$ is any form in $G$.
\end{prop}
It follows from (\ref{pf})
that the characters are not independent but satisfy
\begin{equation}\label{pro2}
\prod_{p|D }N_{F_p/k}\big(\chi_p(Q)\big)=1.
\end{equation}

\subsubsection*{Remark} 
The product formula (\ref{pf}) is equivalent to Weil reciprocity for $k(t)$.  For $a=p$ and $b=q$ with $p,q$ distinct primes in $k[t]$, it's truth  is obvious when it is written in the form
\[
(-1)^{|p||q|}\prod_{\b;\;q(\b)=0}p(\b)=\prod_{\a;\,p(\a)=0} q(\a).
\]
Weil reciprocity can be used to prove quadratic reciprocity (\ref{art}) when $k$ is finite (see \cite{Ros}).
%
%\subsubsection*{Rational classes and Witt rings}

Say two forms $Q,Q'\in \mathcal{Q}_D$ are {\it rationally equivalent} if (\ref{e2}) holds for some $M\in \mathrm{GL}_2(k(t)).$ Such an $M$ must have $\det M=\pm1.$
\begin{prop}\label{mt0}
Two forms $Q,Q'\in \mathcal{Q}_D$ are in the same genus if and only if they are rationally equivalent.
\end{prop}

The proof of  the corresponding result over $\Z$  given in \cite{BS}  is based on the Hasse-Minkowski theorem in the important special case given by Legendre. Most proofs of Legendre's theorem are based on some sort of reduction method and  rely on the 
fact that there are only finitely many integers of a bounded absolute value. In order to prove Proposition \ref{mt0} we must use a method that is different.   The  result we need is a determination of  the Witt ring of $k(t)$ \cite{Wit},  obtained in the form we need  by Milnor \cite{Mil} (c.f.\cite{Kne}, \cite{MH}).  A useful exposition of this result can be found in \cite{Lam}. I will give here a consequence that is easy to apply.

For a field $F$ let $W(F)$ be its Witt ring.  See   \cite[Chap. II]{Lam} for its definition and basic properties (assuming that the characteristic of $F$ is not 2).
We define a set of maps between Witt rings as follows. For each prime $p\in k[t]$ let
\[
\psi(p): W(k(t))\rightarrow W(F_p)
\]
be defined through
\[
\psi(p)\langle a\rangle=\begin{cases} \langle \bar{u}\rangle, \;\text{if $\ell$ is odd}\\0,\;\text{otherwise,}\end{cases}\;\;\;\text{where $ a=p^\ell u$ with $v_p(u)=0$.}
\]
Also, define 
$
\psi(\infty): W(k(t))\rightarrow W(k)
$
 by
\begin{equation}\label{psi}
\psi(\infty)\langle a\rangle=\begin{cases} \langle \bar{a}_0\rangle, \;\text{if $\ell$ is even}\\0,\;\text{otherwise}\end{cases}\;\;\;\text{when $a(t^{-1})=t^{-\ell}(a_0+a_1t+\dots).$}
\end{equation}
Here $\bar{a}$ gives  the natural image of $a\in k(t)^{\sbullet}$  in $F_p^{\sbullet}/F_p^{\sbullet 2}$ (or of $a\in k^{\sbullet}$ in $k^{\sbullet}/k^{\sbullet 2}$ when $p=\infty$). Then the following result is a consequence of \cite[Milnor's Thm. p. 306]{Lam}:
\begin{prop}\label{mh}
The map $\psi=\psi(\infty) \times \prod_p \psi(p)$ induces  an isomorphism 
\[
W(k(t))\rightarrow  W(k)\times  \prod_{p<\infty} W(F_p).
\]

\end{prop}

Proposition \ref{mt0} follows from the next proposition and the definition of genera from the previous subsection.
\begin{prop}\label{has} Let $D\in k[t]$ be square-free and negative with $|D|$ odd.
Suppose that the forms  $Q=(a,b,c)$ and $Q'=(a',b',c')$ are in $\mathcal{Q}_D$ with $a,a'$ prime to $D$.  Then $Q$ and $Q'$ are equivalent over $k(t)$ if and only if 
\[
(a,D)_p=(a',D)_p
\]
for all  $p|D$. 
\end{prop}

\begin{proof}

Each $Q=(a,b,c)\in \mathcal{Q}_d$ is equivalent over $k(t)$ to $\langle a,-\tfrac{D}{a}\rangle\defeq (a,0,-\tfrac{D}{a})$.
 We will use Proposition \ref{mh} to show that $\langle a,-aD\rangle$ and $\langle a',-a'D\rangle$
are Witt-equivalent if and only if $(a,D)_p=(a',D)_p$ for all $p.$  This is enough since as forms over $k(t)$ they have the same dimension, namely 2 (see e.g. \cite[Prop.1.4, p.29]{Lam}).

For $p<\infty$ we have
\begin{equation}\label{bs1}
\psi(p)\langle a, -aD\rangle=\begin{cases} \langle -\bar{u}\bar{a}\rangle, \;\text{$D=pu$ }\\0,\;\text{otherwise}\end{cases}\;\;\; (a,D)_p=\begin{cases} \bar{a}, \;\text{ $p|D$ }\\1,\;\text{otherwise.}\end{cases}.
\end{equation}
In both statements we use that $D$ is a square modulo any $p$ with $p|a;$ 
 in the first  $\psi(p)\langle a, -aD\rangle$ is isotropic and in the second $(a,D)_p=1.$
Note that $u$ in (\ref{bs1}) only depends on $D$.
Also, we have
\begin{equation}\label{bs2}
\psi(\infty)\langle a, -aD\rangle= \langle 1\rangle,\;\;\;(a,D)_\infty=1.
\end{equation}
Here the first statement follows using $a\succ0$ and checking each case when $|a|$ is odd or even, referring to  (\ref{psi}).
That $(a,D)_\infty=1$ was shown in Lemma \ref{lwd}.
Comparison of both sides of (\ref{bs1}) and (\ref{bs2}) for each prime $p$ and $\infty$, together with Proposition \ref{mh}, completes the proof.  
\end{proof}

Now we apply Proposition \ref{mt0} to prove the following analogue of Gauss's principal genus theorem.

\begin{prop}\label{prg}
The principal genus coincides with  $\Cc_D^2$ and hence  $\mathcal{G}_D=\Cc_D/\Cc_D^2.$
\end{prop}

\begin{proof}
First suppose that the class of $Q$ is the square of the class of  $Q'=(a',b',c'),$ where $a'$ is chosen to be prime to $D$. By ii) of Proposition \ref{pro} we have that $Q$ properly represents $a'^2$ and so $\chi_p(Q)=1$ for all $p|D$, hence $Q$ is in the principal genus. 

Conversely, suppose that $Q=(a,b,c)$ is in the principal genus.  By Proposition \ref{mt0} we know that $Q$ is rationally equivalent to $(1,0,-D)$, so there are $q_1,q_2,r\in k[t]$  with $r\neq 0$ 
such that 
\[
a(\tfrac{q_1}{r})^2+2b(\tfrac{q_1}{r})(\tfrac{q_2}{r})+c(\tfrac{q_2}{r})^2=1.
\]
Hence $Q$ represents $r^2$, but perhaps improperly.  
In any case, by Lemma \ref{repp} there are $s,b',c'\in k[t]$ with $s\succ0$ so that $Q$ is properly equivalent to $(s^2,b',c')$ where $D=b'^2-s^2c'$.  If $b'$ and $s$ share a prime factor $q$, then $q^2$ divides both $s^2$ and $b'^2$ so must also divide $D$.
But $D$ is square-free. Thus $s$ is prime to $b'$.
By  Proposition \ref{clg}, the form $Q'=( s,b',c's)\in \Qc_D$ will have the property that the square of its class is the class of $Q$.

\end{proof}

\subsection*{The class group  $\Cc_D$ and the Mordell-Weil group $\mathcal{J}_D$ }\label{Hy}

Let  $D\in k[t]$  be square-free  and  negative with  odd degree $|D|=2g+1.$
In this section I will show that $\Cc_D$ is isomorphic to $\mathcal{J}_D$ and so, in particular, is finitely generated as a consequence of the Mordell Weil theorem.
%We will assume for the proof that  $D$ is monic, which is not a loss of generality.

Let $\mathcal{H}$ be a complete smooth curve with function field $k(t,s)$, where
\[
s^2=D(t).
\]
The function $t$ on $\mathcal{H}$  has a double pole that we denote $\infty$.
Let $Q=(a,b,c)$ represent a class  $C\in \Cc_D$.
Then since $a\succ0$ we can write
 \begin{equation}\label{A}
a(t)=\a^2\prod_{i=1}^m(t_i-t)^{n_i}\end{equation} with $\a\in k^{\sbullet}$ and distinct $t_i\in \bar{k}$ and  $n_i\in \Z^+$. Set $s_i=b(t_i)$. Let $\mathrm{div}(1,0,-D)=0$ and otherwise
associate to $Q$   the degree zero divisor on $\mathcal{H}$  given by
\begin{equation}\label{div}\mathrm{div}(Q)= \sum_{i=1}^m n_i(t_i,s_i)-N\infty,
\end{equation}
where $N=\sum_{i=1}^m n_i$.
  It is defined over $k$ and is semi-reduced in that 
if
$(t,s)$ occurs in $\mathrm{div}(Q)$, then $(t,-s)$ does not, unless $s=0$, in which case the multiplicity of $(0,t)$ is one.

As before, let  $\J_d$ be the  group of  $k$-rational  divisor classes  of degree 0 on  $\mathcal{H}$.   We want to prove the following.
\begin{prop}\label{the1}
The map $Q\mapsto \mathrm{div}(Q)$ induces an isomorphism of groups from $\Cc_D$ onto  $\J_D$.   
\end{prop}

\begin{lemma}\label{l2}
The map $Q\mapsto \mathrm{div}(Q)$ induces a map $\phi$ from $\Cc_D$ to $\J_D$. 
\end{lemma}
\begin{proof}
We must show that $Q\sim Q'$ implies that $\mathrm{div}(Q)$ is equivalent to $\mathrm{div}(Q')$.
Let $a(t)$ be as in (\ref{A}). 
If $Q'$ results from $Q$ by a  translation  $Q'=Q|T_n$  with $n\in k[t]$, we see that $\mathrm{div}(Q)= \mathrm{div}(Q')$ since from (\ref{eqf1}),  \[b(t_i)+n(t_i)a(t_i)=b(t_i). \]  By Lemma \ref{sl}, we are reduced to showing 
the needed equivalence when $Q'=Q|S=(c,-b,a)$. As in the proof of Lemma \ref{cl}, we know that $c\succ0$, so we have 
\[c(t)=\g^2 \prod_{j=1}^{m'}(t_j'-t)^{n'_j}\] with $\g\in k^{\sbullet}$ and $N'=\sum_{j=1}^{m'} n'_j$, where $n_j'\in \Z^+.$  Thus by (\ref{div}) we need to show that
\[
\mathrm{div} (a,b,c)=\sum_{i=1}^m n_i(t_i,b(t_i))-N\infty\;\;\;\text{and}\;\;\;\mathrm{div} (c,-b,a)=\sum_{j=1}^{m'}n_j'(t'_j,-b(t'_j))-N'\infty
\]
are equivalent.
%\[
%\mathrm{div} (c,-b,a)=\sum_{j=1}^{m'}n_j'(t'_j,-b(t'_j))-N'\infty.
%\]
Define the rational function $f$ on $\mathcal{H}$  by
\[
f=\frac{c(t)}{s-b(t)}=-\frac{s+b(t)}{a(t)}.
\]
Now we will show that \[(f)=\big(\mathrm{div} (c,-b,a)-\mathrm{div} (a,b,c)\big)\] where, as usual, $(f)$ is the divisor of $f$.  This will prove the claim. 

 Suppose first  that $b(t_i)\neq 0$ and $b(t_j')\neq 0.$ Clearly  $( t_i,b(t_i))\neq (t_j',- b(t'_j))$.  The factors \[s-b(t)\;\;\mathrm{ and} \;\;c(t)\] have  zeros of order $n_i$ at $(t_i,b(t_i))$ and  $n_j'$ at $(t_j',\mp b(t'_j))$, respectively. Of the latter, the zeros at $( t'_j,b(t'_j))$ are cancelled by the factor $s-b(t).$
% For this note that \[s^2-B(t)^2=-a(t)C(t).\]
  A similar argument works for any points $(t_i,0)$  or $(t'_j,0)$ that occur. 
 Since $d$ is square-free, $(t_i,0)\neq (t_j',0)$.   
 
 Finally, the behavior at infinity also
matches since $t$ has a double pole at $\infty$.

\end{proof}

\begin{lemma}
The map  $\phi: \Cc_D\rightarrow \J_D$ is a group homomorphism. 
\end{lemma}
\begin{proof}
By Lemma \ref{repp} and Proposition \ref{clg},  any two classes $C,C'\in \Cc_D$ contain forms $Q=(a,b,ca')$ and $Q'=(a',b,ac)$, resp., where $a,a'$ are coprime and prime to $D$,
and the composition of $C$ and $C'$ contains the form $Q''=(aa',b,*)$.  Thus $\mathrm{div}(Q)+\mathrm{div}(Q')= \mathrm{div}(Q'').$
\end{proof}
Proposition \ref{the1} now follows from the next lemma. 
\begin{lemma}
The map  $\phi: \Cc_D\rightarrow \J_D$  is bijective.
\end{lemma}
\begin{proof}
By Proposition \ref{min}, every class $C\in \Cc_D$ contains a unique form $Q=(a,b,c)$ where
 \[|b|<|a|\leq g,\] with $(-1)^{|a|}a$  monic. By the Riemann-Roch theorem (see e.g. \cite{Lan0}) and a straightforward argument (c.f. \cite[pp. 96-97]{Can1}), every divisor class has a unique representative of the form $\mathrm{div}(Q)$ for such a $Q$.
 Here $b$ is  uniquely determined by the condition $|b|<|a|$.
\end{proof}

\begin{prop}\label{mw}
For $k$ a number field, $D\in k[t]$  square-free and negative of odd degree, the group $\Cc_D$ is finitely generated.
\end{prop}
\begin{proof}
This follows from Proposition \ref{the1} and the Mordell-Weil Theorem 
\cite{Wei}, \cite{Lan}.
\end{proof}

\subsubsection*{Remark}
As already mentioned,  that a result like Proposition \ref{the1} should hold can be inferred from the paper of D. Cantor \cite{Can1} (see also \cite{Mum}). However, the  isomorphism we give does not seem to be well-known.

\subsection*{Proof of  Theorem \ref{t0}}
i) ($\J_q$ finite implies that (\ref{eqt1}) holds.) Suppose that $\mathcal{J}_q$ is finite. By Proposition \ref{the1} we know that $\Cc_q$ is finite. Since $q$ is a prime, 
it follows from Proposition \ref{amb} that  $\Cc_q$ has no classes of positive even order and so any class $C\in\Cc_q$ is in $\Cc_q^2$.   Let $p$ be any prime different from $q$.
If $\big(\tfrac{q}{p}\big)=1$ then $q=b^2+pc$ for some $b,c,$ so $Q=(p^*,b,\pm c)\in \Qc_{q}$.  Thus since $C\in\Cc_q^2$, we have that $\chi_q(Q)=(p^*,q)_q=1$ and $\big(\tfrac{p^*}{q}\big)=1$.

ii) ($\J_q$ infinite implies (\ref{eqt1}) does not hold.)
By  Proposition \ref{the1} again, if $\mathcal{J}_q$ is infinite so is  $\Cc_q$ and by Proposition \ref{mw} it must contain a class $C$ not in $\Cc_q^2.$ Let $Q(x,y)=(a,b,c)\in C$,  so $b^2-ac=q.$ By Proposition \ref{min} and the proof of Lemma \ref{cl}, we may assume that $|b|<|a|<|c|$
with $c^*$ monic. 
The polynomial
\[
F(x,t)=a(t)x^2+2b(t)x+c(t)
\]
is irreducible in $k[y,t]$.  To see this, observe that $\gcd(a,b,c)=1$ and, as a quadratic in $x$, the discriminant of $F$ is $4q(t),$ which is not a square.
By Hilbert's irreducibility theorem \cite{Hil}, \cite{Lan} there are infinitely many  $x\in k$ such that 
$F(x,t)=Q(x,1)$ is irreducible in  $k[t].$ For any  such  $x$, it  follows from Lemma \ref{repp} that $C$ contains the form $Q'=(p^*, b',c')$, where  \[p(t)=(-1)^{|c|}Q(x,1)\neq q(t)\] is prime. Thus $q=b'^2-p^*c'$ and so $\big(\tfrac{q}{p}\big)=1.$  
Since $C\not\in \Cc_q^2$, by the principal genus theorem Proposition \ref{prg}, we must have that $\chi_q(Q')=(p^*,q)_q\neq 1.$ Thus $\big(\tfrac{p^*}{q}\big)\neq1$.
\qed

\subsubsection*{Remark} It is easy to modify the above argument to produce infinitely many primes $p$  of either odd or even degree with $\big(\tfrac{p^*}{q}\big)\neq1$. 

\subsection*{Proof of  Theorem \ref{t1}}
By Proposition \ref{the1}, 
we require a  characterization of those classes in $\Cc_D$ of order at most 2 that are contained in $\Cc_D^2$. 

\smallskip\noindent
 Claim: the number of such classes is given by
$2^{e_2(\J)}$, where as before we write $\J=\J_D.$
To see this, let  $m$  be the number of primary cyclic subgroups  of order 2.    The number of classes $C$  for which $C^4=1$ is \begin{equation}\label{las}2^{m}4^{e_2(\J)}=2^{m+e_2(\J)}2^{e_2(\J)}=2^{e_1(\J)}2^{e_2(\J)}=\#\,(\text{classes of order at most 2})\times 2^{e_2(\J)}.\end{equation}
For each class in $C_2\in \Cc_D$ of order at most 2  contained in $\Cc_D^2$,  choose a fixed $C_1\in \Cc_D$ with $C_1^2=C_2$.  By multiplication of each $C_1$ by all classes of order at most 2, we get all   classes $C$ with $C^4=1.$
The claim follows by a comparison with (\ref{las}). 

  Proposition \ref{prg} implies that $2^{e_2(\J)}$ is the number of classes of order at most 2 that are contained in the principal genus of $\Cc_D.$
By (the proof of) Proposition \ref{amb}, these classes  are  in one-to-one correspondence with the   ambiguous forms $(D_1,0,-D_2)$ that are killed by all genus characters, where $D=D_1D_2$ is as in the statement of the theorem. Using the definition of the genus characters given in (\ref{gc}), 
these forms correspond to the decompositions described in Theorem \ref{t1}.

\qed
\section{Binary cubic forms over $k[t]$}\label{s4}

\subsection*{Preliminaries}
Many of the general identities and transformation properties of binary cubic forms over $\Z$  can be readily adapted  to forms with coefficients in $k[t]$. 
Some references for the classical theory are \cite{Hilb} and \cite{Sa}. 
For $a,b,c,d\in k[t]$ let
\[
F(u,v)=au^3+3bu^2v+3cuv^2+dv^3=(a,b,c,d)(u,v).
\]
Its discriminant is 
\[
\D=\D_F=a^2d^2-6abcd+4b^3d+4ac^3-3b^2c^2=(ad-bc)^2-4(ac-b^2)(bd-c^2).
\]
The Hessian of $F$ is
\begin{equation}\label{hes}
H_F(u,v)=(ac-b^2)u^2+(ad-bc)uv+(bd-c^2)v^2=(ac-b^2, \tfrac{1}{2}(ad-bc),bd-c^2)(u,v)
\end{equation}
and the cubic covariant is
\[
J_F(u,v)=(2b^3-3abc+a^2d,b^2c-2ac^2+abd,-bc^2+2b^2d-acd,-2c^3+3bcd-ad^2).
\]
The discriminant of $H_D$ is $D=\frac{1}{4}\D_F$.
The group  $\mathrm{SL}_2(k[t])$ acts  as expected on $F, H_F, J_F$: thus
 for $M=\left( \begin{smallmatrix}m& r\\n&s \end{smallmatrix}\right)$  we have   $F|M(x,y) =F(mx+ny,rx+sy)$.
 The discriminant $\D_F$ is invariant and 
the forms $H_F$ and $J_F$ are covariants of $F$, meaning that $H_{F|M}=H_F|M$ and $J_{F|M}=J_F|M$.
The principal form of discriminant $\D=4D$ 
\begin{equation}\label{Fp}
F_{0}=(0,1,0,D)
\end{equation}
has $-H_0\defeq -H_{F_0}=(1,0,-D)$ and $\tfrac{1}{2}J_{0} \defeq \tfrac{1}{2}J_{F_0}=(1,0,D,0)$.

\subsection*{The syzygy and its converse}
After Eisenstein \cite{Eis} and Cayley \cite{Cay},  a cubic form and its covariants  satisfy the syzygy
\[
J_F^2=-4H_F^3+\D F^2,
\]
which can be checked directly.
Setting $X=-H_F$, $Y=\tfrac{1}{2}J_F$ and $Z=F$ we have
\begin{equation}\label{ssyz}
Y^2=X^3+DZ^2,
\end{equation}
where $D=\tfrac{1}{4}\D$ is the discriminant of $H_F$.
The following converse result is crucial. Over $\Z$ the corresponding result  is due to Arndt \cite{Arn}   and Mordell \cite{Mor0}, \cite[p.216]{Mor}.

\begin{prop}\label{momo}
Suppose that $D\in k[t]$ is square-free. If $X,Y,Z\in k[t]$ satisfy $\gcd(X,Z)=\gcd(Y,Z)=1$ and 
\begin{equation*}\label{mo1}
Y^2=X^3+DZ^2,
\end{equation*}
then there exists a binary cubic form $F=(Z,b,c,d)$ with $b,c,d\in k[t]$ and discriminant $4D$ for which $-H_F=(X,*,*)$ and $\tfrac{1}{2}J_F=(Y,*,*,*).$ \end{prop}
\begin{proof}
First note that $\gcd(X,Y)=1$ since $D$ is square-free. Thus both $X$ and $Y$ are prime to $D$.
From (\ref{ssyz}) we see that $D$ is a square modulo $X$ so  there exists a binary form $(X,B,C)$ with $B^2-XC=D$.   We have that $\gcd(Z,X)=1$ and from (\ref{ssyz}) we may suppose that \begin{equation}\label{supp}ZB\equiv -Y\pmod{X^3}.\end{equation}
The class of $(X,B,C)$ is uniquely determined once we force $B$ to satisfy condition (\ref{supp}).
 Clearly $X\neq 0$.  After Arndt \cite{Arn}, set
\begin{equation}\label{cc}
b=\tfrac{1}{X}(Y+BZ)\;\;\;c=\tfrac{1}{X^2}(DZ+2YB+ZB^2)\;\;\;d=\tfrac{1}{X^3}(YD+3BDZ+3YB^2+ZB^3).
\end{equation}
A calculation shows that the discriminant of $F=(Z,b,c,d)$ is $4D$, that $-H_F=(X,B,C)$ and that $\tfrac{1}{2}J_F=(Y,*,*,*).$

We need to show that $b,c,d\in k[t].$
Clearly $b\in k[t].$ 

To see that $c\in k[t],$
consider
\begin{align*}
U\times V=&(DZ+2YB+ZB^2)(DZ-2YB+ZB^2)\\&=Z^2(D+B^2)^2-4Y^2B^2\\
&\equiv Z^2(D+B^2)^2-4DZ^2B^2 \pmod{X^3} \;\;\text{by (\ref{ssyz})}\\&=Z^2(B^2-D)^2\equiv 0 \pmod{X^2}.
\end{align*}
Now $U$ and $V$ cannot share any primes with $X$ since $U-V=4YB.$
That $c\in k[t]$ will follow  if we can show that $X|ZU.$ From $X|(ZB+Y)$ we have that \[ZU\equiv Z(DZ+YB)\equiv Y(Y+ZB)\equiv 0 \pmod{X}.\]

Finally, from (\ref{cc}) we must prove that 
\[
X^3d=YD+3BDZ+3YB^2+ZB^3\equiv 0\pmod{X^3}.
\]
Using (\ref{supp}) we have that 
\[
Z^2X^3d\equiv Z^2YD-3YDZ^2+(3Y^3-Y^3)\equiv 2Y(Y^2-DX^2)\equiv 0\pmod {X^3},
\]
as desired. 
\end{proof}

%\subsection*{Negative discriminants}
%\subsubsection*{Remarks}
%i) The class of $-H_F$ is uniquely determined by the choice of $B$ made to satisfy condition (\ref{supp}).
%

%ii) One must take care when applying congruence arguments that are familiar over $\Z$ to $k[t].$ For instance, one cannot always conclude from $(\tfrac{x^3}{p})=1$ that $(\tfrac{x}{p})=1$ 
%for $x,p\in k[t]$ with $p$ prime and $p\nmid x.$ This implication  is used in the proof of the $\Z$-version of Proposition \ref{momo} given in \cite{Mor}.
%

\begin{prop}\label{cayy}
Suppose that $D\in k[t]$ is square-free and negative with odd degree. Let $F=(a,b,c,d)$ with $a,b,c,d\in k[t]$ have discriminant $4D$. Then $Q=-H_F=(A,B,C)\in \Qc_D$ and its class  has order 1 or 3 in $\Cc_D.$ 
The only other cubic form with Hessian $-Q$ is $-F$. 
\end{prop}
\begin{proof}
To show that $Q\in \Qc_D$, first note that the discriminant of $Q$ is $D$. 
We also must have $A\succ0$. To see this we can apply (\ref{ssyz}).  First,  if $|A|$ is even we need the leading coefficient $A_0$ of $A$ to be in $k^{\sbullet 2}$.   Since $DZ^2$ has odd degree we have $A_0^3=Y_0^2$ where $Y_0$ is  the leading coefficient of $Y$,
so $A_0=(\frac{Y_0}{A_0})^2.$ 
  If $|A|$ is odd the leading coefficients of $A^3$ and $DZ^2$  must cancel so again $A_0\in k^{\sbullet 2}$.

Next we show that the class of $Q=(A,B,C)$ has order 1 or 3. 
We have the following identity  (from \cite{Cay1}) that can be verified directly:
\begin{align}\label{ord3}
&(A,-B,C)(buu'+cuv'+cu'v+dvv',auu'+buv'+bu'v+cvv')\\\nonumber&=(A,B,C)(u,v)\times (A,B,C)(u',v'),
\end{align}
 where $A=b^2-ac,B=bc-ad$ and $C=c^2-bd.$
Thus by Proposition \ref{clg2} and ii) of  Proposition \ref{clg} we have  $Q\otimes Q=Q^{-1}$, so that the class of $Q$ has order 1 or 3 in $\Cc_D$. 
%To see that $P$ from (\ref{ord3}) needed in  Proposition \ref{clg2} is primitive, use that $\gcd(A,B,C)=1$, which follows since $D$ is square-free. 

The proof of the final statement is also adapted from \cite[$2^{nd}$ letter]{Cay1}, but the argument  required some reworking to make it more understandable.
Suppose that $F'=(a',b',c',d')$ has $-H_{F'}=(A,B,C).$ By (\ref{ord3}) we have
\[
(A,-B,C)(x,y)=(A,-B,C)(x',y')
\]
where
\begin{align}\label{four}
x&=buu'+cuv'+cu'v+dvv'\\\nonumber
y&=auu'+buv'+bu'v+cvv'\\\nonumber
x'&=b'uu'+c'uv'+c'u'v+d'vv'\\\nonumber
y'&=a'uu'+b'uv'+b'u'v+c'vv'.
\end{align}
Also by assumption
\begin{equation}\label{un}
A=b^2-ac=b'^2-a'c'\;\;\;\;B=bc-ad=b'c'-a'd'\;\;\;\;C=c^2-bd=c'^2-b'd'.
\end{equation}
First we  will show that 
\begin{equation}\label{ver}
(x',y')^t=M(x,y)^t \;\;\;\;\text{for}\;\;\;M=\left( \begin{smallmatrix}m& r\\n&s \end{smallmatrix}\right)\in \mathrm{SL}_2(k(t)).
\end{equation}
After verifying this, we will  show that in fact we may take $m,r,n,s\in k[t]$.
Given this,  the uniqueness of $F$ up to sign follows from  Lemma \ref{ls1}, which implies  that $Q$ has only trivial automorphs. 

A calculation shows that if
\begin{equation}\label{ano}
m=\tfrac{bb'-ac'}{A}\;\;\;\;r=\tfrac{-cb'+bc'}{A}\;\;\;\;n=\tfrac{ba'-ab'}{A}\;\;\;\;s=\tfrac{-ca'+bb'}{A}
\end{equation}
then $\det M=1$. Here we use the first formula of (\ref{un}) and that $A\neq 0.$
Another calculation using (\ref{ano}) gives that 
\begin{align*}
&bm+ar=b'\;\;\;\;bn+as=a'\\
&cm+br=c'\;\;\;\;cn+bs=b'.
\end{align*}
To verify (\ref{ver}) it is now enough to show that $dm+cr=d'$ and $dn+cs=c'$.  By (\ref{ano}) and (\ref{un})
\[
d(bb'-cc')+c(-cb'+bc')=-Cb'+c'B=b'^2d'-a'c'd'=d'A,
\]
which gives $dm+cr=d'$. That  $dn+cs=c'$ follows  similarly.

Finally, we must show that $m,r,n,s\in k[t].$ In case $B=0$ we have from (\ref{four}) and  (\ref{ver}) that
\[
cm+br=c',\;\;\;dm+cr=d'\;\;\;\;\text{and}\;\;\;cn+bs=b',\;\;\;dn+cs=c',
\]
hence that
 \begin{equation*}
m=\tfrac{cc'-bd'}{C}\;\;\;\;r=\tfrac{-dc'+cd'}{C}\;\;\;\;n=\tfrac{cb'-bd'}{C}\;\;\;\;s=\tfrac{-db'+cc'}{C}.
\end{equation*}
 Since $\gcd(A,C)=1$, from this and (\ref{ano})  we must have that $m,r,n,s\in k[t].$
 If $B\neq 0$ we argue similarly by giving a third set of fractions from (\ref{four}) and  (\ref{ver}) for $m,r,n,s$, now with denominator $B$,
 starting with
 \[
 bm+ar=b',\;\;\;dm+cr=d'\;\;\text{and}\;\;bn+as=a',\;\;\;dn+cs=c'.
 \]
 
\end{proof}
 Propositions \ref{cayy} and \ref{mw} immediately imply the following result, since the number of classes in $\Cc_D$ of order at most three is finite. 

\begin{prop}\label{ii}
The class number of binary cubics over $k[t]$ with discriminant $4D$ is finite, when $D\in k[t]$ is negative,   square-free and of odd degree.
\end{prop}
%\begin{proof}
%This is an immediate consequence of  Propositions \ref{cayy} and \ref{mw}, since the number of classes in $\mathcal{J}_D$ of order at most 3 is finite and the number of classes of binary cubics
%\en{proof}
%

\subsection*{Elliptic curves over $k(t)$}

Consider the elliptic curve with $j$-invariant zero over $k(t)$ defined by
\begin{equation}\label{Ed}
\mathcal{E}_D: y^2=x^3+D(t)
\end{equation}
where $D\in k[t]$ is square-free.   
Recall that $D'=-27D(-3t)$, which  is also square-free and that if  $D$ is negative with odd degree so is $D'$.    
By the Mordell-Weil theorem over function fields (see \cite[p.230]{Sil})  the rational points on $\mathcal{E}_D$ and $\mathcal{E}_{D'}$ form finitely generated abelian groups.

We have the pair of dual 3-isogenies $\psi:\mathcal{E}_D\rightarrow \mathcal{E}_{D'}$ and  $\psi':\mathcal{E}_{D'}\rightarrow \mathcal{E}_{D}$ given by
\begin{align}\nonumber
\psi: (x,y)\mapsto& \Big(\big(\tfrac{x^3+4D}{x^2}\big)^\s,\big(\tfrac{y(x^3-8D)}{x^3}\big)^\s\Big)\;\;\;\;\text{and}\\\label{phip}\psi': (x',y')\mapsto& \Big(\big(\tfrac{x'^3+4D'}{9x'^2}\big)^{\s^{-1}},\big(\tfrac{y'(x'^3-8D')}{27x'^3}\big)^{\s^{-1}}\Big),
\end{align}
where $\s:k(t)\rightarrow k(t)$ is the automorphism defined by $t\mapsto -3t.$
Then $\psi'\circ \psi$ and $\psi\circ \psi'$ give the tripling maps on $\mathcal{E}_D$ and $\mathcal{E}_{D'}$, respectively.

Suppose that $D$ is negative and of odd degree. 
Given a (finite) rational point $P=(x,y)$ on $\mathcal{E}_D$ we can write $(x,y)=(\tfrac{X}{V^2},\tfrac{Y}{V^3})$ where $V\in k[t]$ is monic and $\gcd(X,V)=\gcd(Y,V)=1.$
 Taking  $Z=V^3$,  we associate to $P$ the class $C_P$ of any form $Q=-H_F$ that, together with $F$ and $J_F$,  gives rise to $P$ via Proposition \ref{momo}. 
 Note that the class of $Q$ is uniquely determined by (\ref{supp}). Thus by Proposition \ref{cayy}, $C_P$ has order 1 or 3 in $\Cc_D$.  Similarly we associate to any point $P'=(x',y')$ on $\mathcal{E}_{D'}$ the class $C_{P'}$, which has order 1 or 3 in $\Cc_{D'}$.

\begin{lemma}\label{triv}
The class $C_P\in \Cc_D$ has order 1 if and only if $P$ is the image of some $Q'\in \mathcal{E}_{D'}$ under $\psi'.$ The class $C_{P'}\in \Cc_{D'}$ has order 1 if and only if $P'$ is the image of some $Q\in \mathcal{E}_{D}$ under $\psi.$
\end{lemma}
\begin{proof}
Note that $F(u,v)=V^3$ if and only if $F(\tfrac{u}{V},\tfrac{v}{V})=1.$
Recall the principal form $F_0$ from (\ref{Fp}). It is readily checked that there is a bijection between the set 
\begin{equation}\label{bij}
\{(x,y)\in k(t)^2; F_0(x,y)=3x^2y+Dy^3=1\} 
\end{equation}
and the finite rational points $(x',y')$   on $\mathcal{E}_{D'}$ given by 
 \begin{equation}\label{firi}(x',y')=\big((\tfrac{3}{y})^\s,(\tfrac{9x}{y})^\s\big)\;\;\;\text{and}\;\;\;(x,y)=\big((\tfrac{y'}{3x'})^{\s^{-1}},(\tfrac{3}{x'})^{\s^{-1}}\big).\end{equation}
A computation using (\ref{phip}) and (\ref{Fp}) shows that for $(x,y)$ with $3x^2y+Dy^3=1$ and $(x',y')$ from (\ref{firi}) we have 
\[
\psi'(x',y')=\big(-H_0(x,y),\tfrac{1}{2}J_0(x,y)\big).
\]
The first statement follows since we have from Proposition \ref {cayy} that $-H_F$ is principal if and only if $F$ is principal.

The proof of the second statement is similar, using that for $F_0'=(0,1,0,D')$ there is  a bijection between the set
\[
\{(x',y')\in k(t)^2; F'_0(x',y')=3x'^2y'+D'y'^3=1\} 
\]
and the finite rational points $(x,y)$ on $\mathcal{E}_D$ from (\ref{Ed}) 
given by 
 \[(x,y)=\big((\tfrac{1}{3y'})^{\s^{-1}},(\tfrac{x'}{3y'})^{\s^{-1}}\big)\;\;\;\text{and}\;\;\;(x',y')=\big((\tfrac{y}{x})^{\s},(\tfrac{1}{3x})^{\s}\big).\]
 Now use that
\[
\psi(x,y)=\big(-H'_0(x',y'),\tfrac{1}{2}J'_0(x',y')\big),
\]
where $-H_0'=-H_{F_0'}=(1,0,-D')$ and  $\tfrac{1}{2}J_0'=\tfrac{1}{2}J_{F_0'}=(1,0,D',0)$ and apply Proposition \ref {cayy} as before.
\end{proof}

\subsection*{Proof of Theorem \ref{t2}}
To prove the  first statement, suppose that $P=(x,y)\in \mathcal{E}_D$. Then
\begin{equation*}
2P=(x',y')=\big(\tfrac{9x^4}{4y^2}-2x,\tfrac{8y^4-36x^3y^2+27x^6}{8y^3}\big).
\end{equation*}
Write $(x,y)=(\frac{X}{V^2},\frac{Y}{V^3})$ where $X,Y,V\in k[t]$ with $\gcd(X,V)=\gcd(Y,V)=1.$ Then we have \begin{equation}\label{cap}Y^2=X^3+DV^6\;\;\;\text{and}\end{equation} 
\begin{equation}\label{doub}
V^2x'=\tfrac{9X^4}{4Y^2}-2X\;\;\;\text{and}\;\;\;V^3y'=\tfrac{8Y^4-36X^3Y^2+27X^6}{8Y^3}.
\end{equation}
Since $D$ is square-free we must have $\gcd(X,Y)=1.$ There are two cases to consider.

i) If $|Y|>0$ then the denominator of $x'$ has larger degree than that of $x$ and so the first statement of Theorem \ref{t2} follows.

ii) If $Y\in k^{\sbullet}$ 
then we may write
$(x',y')=(\frac{X'}{V^2},\frac{Y'}{V^3})$ where $X',Y'\in k[t]$. Here $\gcd(X',V)=\gcd(Y',V)=1$, since by  (\ref{cap}) and (\ref{doub}) we have for any prime $P|V$ 
\begin{align}\nonumber&4Y^2X'=X(9X^3-8Y^2)\equiv XY^2\pmod{P}\;\;\;\text{and}\\\label{sec} &8Y^3Y'=8Y^4-36X^3Y^2+27X^6=27D^2V^{12}-18DV^6Y^2-Y^4\equiv -Y^4\pmod{P},\end{align}
where $Y\neq 0$ and $P\nmid X.$ Since $|D|>0$ we have from (\ref{cap}) that $|X|>0$, hence from the first equation of  (\ref{sec}) that $|Y'|> 0$.
Repeat the argument from the beginning  with   $P=(x',y')\in \mathcal{E}_D$ in place of $(x,y)$.  Now i) applies and again the first statement of Theorem \ref{t2} follows.

Now we prove the (contrapositive of the) second statement of Theorem \ref{t2}.  By the  first statement and Mordell-Weil for $\mathcal{E}_D$,  if $\mathcal{E}_D$ has  a finite rational point, it is of infinite order and we may assume that it is not the triple of any other point.
Let $P$ be such a point. 
If the $C_P\in \Cc_D$ that corresponds to $P$ as above Lemma \ref{triv}  is not trivial we are done, since $P$  is then of order 3.  Otherwise, by Lemma \ref{triv}, {$P$ is the image under $\psi'$ of some point $P'\in \mathcal{E}_{D'}$.
In that case the class $C_{P'}$ in $\Cc_{D'}$ corresponding to $P'$ is of order 3, since otherwise $P'=\psi(Q)$ for some $Q\in \mathcal{E}_D$ and then 
$P=\psi'(\psi(Q))=3Q$, and $P$ is a triple. The second statement now follows since $\J_D$ is isomorphic to $\Cc_D$ and $\J_{D'}$ is isomorphic to $\Cc_{D'}$ by Proposition \ref{the1}.
\qed

\end{document}